\documentclass[12pt]{amsart}
\usepackage{amssymb}
\usepackage{amsmath}
\usepackage{amsthm}
\usepackage{graphicx}
\usepackage{enumerate}
\usepackage{tikz}
\usepackage{tabularx}
\usepackage{mathdots}
\usepackage{tabmac_avi}
\usepackage{accents}

\newtheorem{thm}{Theorem}[section]
\newtheorem{prop}[thm]{Proposition}
\newtheorem{cor}[thm]{Corollary}
\newtheorem{lem}[thm]{Lemma}

\newtheorem{prob}[thm]{Problem}

\numberwithin{equation}{section}

\newcommand{\sn}{\mathfrak{S}_n}

\newcommand{\mfs}[1]{\mathfrak{S}_{#1}}

\newcommand{\bn}{\mathfrak{B}_n}
\newcommand{\gnd}{\mathcal G}
\newcommand{\mfg}[1]{\gnd(#1)}

\newcommand{\czn}{\mathbb{C}[z_{1,1},\dotsc,z_{n,n}]}

\newcommand{\csn}{\mathbb{C}[\sn]}

\newcommand{\trsp}{\mathcal{T}}

\newcommand{\imm}[1]{\mathrm{Imm}_{#1}}

\newcommand{\gimm}[1]{\mathrm{Imm}_{#1}^{\gnd}}
\newcommand{\bimm}[1]{\mathrm{Imm}_{#1}^{\bn}}
\newcommand{\simm}[1]{\mathrm{Imm}_{#1}^{\sn}}

\newcommand{\grimm}[2]{\mathrm{Imm}_{#1}^{\mfg{\smash{#2}}}}

\newcommand{\glambdamin}{\mfg{\bos \lambda}^-}
\newcommand{\inv}{\mathrm{inv}}

\newcommand{\defeq}{:=} 

\newcommand{\bos}[1]{{\boldsymbol #1}}
\newcommand{\spn}{\mathrm{span}}
\newcommand{\sgn}{\mathrm{sgn}}
\newcommand{\triv}{\mathrm{triv}}

\newcommand{\ntnsp}{\negthinspace}
\newcommand{\ntksp}{\negthickspace}
\newcommand{\nTksp}{\negthickspace\negthickspace}

\newcommand{\bp}{\begin{prob}}
\newcommand{\ep}{\end{prob}}

\newcommand{\perm}{\mathrm{per}}
\newcommand{\permmon}[2]{#1_{1,#2_1} \ntnsp\cdots {#1}_{n,#2_n}}

\newcommand{\upparrow}{\big \uparrow \nTksp \phantom{\uparrow}}

\newcommand*{\dt}[1]{%
  \accentset{\mbox{\large .}}{#1}}
\newcommand*{\ddt}[1]{%
  \accentset{\mbox{\large .\hspace{-0.25ex}.}}{#1}}

\begin{document}
\author{Mark Skandera}
\title{Generating functions for monomial characters of wreath products
  $\mathbb Z/d\mathbb Z \wr \sn$}


\date{\today}

\begin{abstract}
  Let
  $\mathbb Z/d\mathbb Z \wr \sn$ denote the wreath product
  of the cyclic group $\mathbb Z/d\mathbb Z$ with the symmetric group $\sn$.
  We define generating functions for monomial
  (induced one-dimensional) characters of
  $\mathbb Z/d\mathbb Z \wr \sn$
  and express these in terms of determinants and permanents.
  This extends work of Littlewood
  ({\em The Theory of Group Characters and Representations of Groups}, 1940)
  and Merris and Watkins
  ({\em Linear Algebra Appl.}, {\bf 64}, 1985)
  on generating functions for the monomial characters of $\sn$.
  \end{abstract}

\maketitle

\section{Introduction}\label{s:intro}

Let $z = (z_{i,j})$ be an $n \times n$ matrix of variables
and let $\sn$ be the symmetric group.
For each linear functional $\theta: \csn \rightarrow \mathbb C$,
define the generating function
\begin{equation}\label{eq:imm}
  \imm{\theta}(z) \defeq \sum_{w \in \sn} \theta(w) \permmon zw \in \czn
\end{equation}
for $\theta$, and call this the {\em $\theta$-immanant}.
Such functions appeared originally in~\cite[p.\,81]{LittlewoodTGC}
for $\theta$ equal to irreducible $\sn$-characters $\chi^\lambda$,
and were extended in \cite[\S 3]{StanPos} to general $\theta$.
As is the case with many functions, a simple formula for a generating function
for $\theta$ can be as useful as a simple formula for the numbers
$\{ \theta(w) \,|\, w \in \sn \}$ themselves.

Particularly simple generating functions for
the {\em monomial} (induced one-dimensional) characters of $\sn$
are expressed in terms of integer partitions,
ordered set partitions, and
submatrices of $z$.
Call a nonnegative integer sequence 
$\lambda = (\lambda_1,\dotsc,\lambda_r)$
satisfying $\lambda_1 + \cdots + \lambda_r = n$
a {\em weak composition of $n$}
and write $|\lambda| = n$, $\ell(\lambda) = r$.
If the components of $\lambda$ are weakly decreasing and positive,
call it
an {\em (integer) partition of $n$}
and write $\lambda \vdash n$.
For any weak composition $\lambda$ of $n$,
call a sequence $(I_1, \dotsc, I_r)$ of pairwise disjoint subsets
of $[n] \defeq \{1, \dotsc, n \}$ an
{\em ordered set partition of $[n]$ of type $\lambda$}
if $|I_j| = \lambda_j$ for $j = 1,\dotsc,r$.
(We remark that our nonstandard terminology allows empty sets in
set partitions, whereas standard terminology~\cite[pp.\,39, 73]{StanEC1}
does not.)
Given subsets $I$, $J$ of $[n]$, define the {\em $(I,J)$-submatrix} of $z$
to be $z_{I,J} = (z_{i,j})_{i \in I, j \in J}$.

The class function space of $\sn$ has two standard bases consisting of
monomial characters: 
the {\em induced trivial character} basis
$\smash{\{ \eta^\lambda = \triv {\upparrow_{\mfs{\lambda}}^{\sn}} \,| \,
  \lambda\vdash n \}}$
and the {\em induced sign character} basis
$\smash{\{\epsilon^\lambda = \sgn {\upparrow_{\mfs{\lambda}}^{\sn}} \,|\,
  \lambda \vdash n \}}$,
where $\mfs \lambda$ is the Young subgroup of $\sn$ indexed by $\lambda$.
(See, e.g., \cite{Sag}.)
Littlewood~\cite[\S 6.5]{LittlewoodTGC} and Merris and Watkins~\cite{MerWatIneq}
came close to expressing the
$\eta^\lambda$- and $\epsilon^\lambda$-immanants as
\begin{align}
  \imm{\epsilon^\lambda}(z) \label{eq:lmw1}
  &= 
  \nTksp \sum_{(J_1,\dotsc,J_\ell)} \nTksp
  \det(z_{J_1,J_1}) \cdots \det(z_{J_\ell,J_\ell}),\\
  \imm{\eta^\lambda}(z) \label{eq:lmw2}
  &= 
  \nTksp \sum_{(J_1,\dotsc,J_\ell)} \nTksp
  \perm(z_{J_1,J_1}) \cdots \perm(z_{J_\ell,J_\ell}),
  \end{align}
where the sums are over all ordered set partitions $(J_1,\dotsc,J_\ell)$ of $[n]$
of type $\lambda = (\lambda_1,\dotsc,\lambda_\ell)$.
For example, we have
\begin{equation*}
  \begin{aligned}
  \imm{\epsilon^{21}}(z) &=
  \det\ntnsp\begin{bmatrix} z_{1,1} & z_{1,2} \\ z_{2,1} & z_{2,2} \end{bmatrix} \ntnsp z_{3,3} \,+\,
  \det\ntnsp\begin{bmatrix} z_{1,1} & z_{1,3} \\ z_{3,1} & z_{3,3} \end{bmatrix} \ntnsp z_{2,2} \,+\,
  \det\ntnsp\begin{bmatrix} z_{2,2} & z_{2,3} \\ z_{3,2} & z_{3,3} \end{bmatrix} \ntnsp z_{1,1}\\
  &= 3 z_{1,1}z_{2,2}z_{3,3} - z_{1,2}z_{2,1}z_{3,3} - z_{1,3}z_{2,2}z_{3,1} - z_{1,1}z_{2,3}z_{3,2},
  \end{aligned}
\end{equation*}
and $\epsilon^{21}(123) = 3$, $\epsilon^{21}(213) = \epsilon^{21}(321) = \epsilon^{21}(132) = -1$, $\epsilon^{21}(312) =  \epsilon^{21}(231) = 0$.
While
Littlewood, Merris, and Watkins may not have written
Equations (\ref{eq:lmw1}) -- (\ref{eq:lmw2})
explicitly, we call them the {\em Littlewood--Merris--Watkins identities}.
These identities have played an important role in
the evaluation of (type-$A$) Hecke algebra characters at Kazhdan--Lusztig
basis elements~\cite{CHSSkanEKL},
\cite{CSkanTNNChar}, \cite{KLSBasesQMBIndSgn},
the formulation of a
generating function for irreducible Hecke algebra characters~\cite{KSkanQGJ},
and the interpretation of coefficients of chromatic symmetric
functions~\cite{CHSSkanEKL}, \cite{SkanCharChrom}.
The identity in our main result (Theorem~\ref{t:main}) plays an important
role in the evaluation of hyperoctahedral group characters at elements
of the type-$BC$ Kazhdan-Lusztig basis~\cite{SkanGCEHGC}.

Let $\gnd = \gnd(n,d)$
be the wreath product $\mathbb Z/d\mathbb Z \wr \sn$.
Its class function space has $2^d$ standard bases consisting of monomial
characters, and it is possible to
use a matrix of $dn^2$ variables to construct generating
functions 
analogous to (\ref{eq:lmw1}) -- (\ref{eq:lmw2})
for the elements of these bases.
In Section~\ref{s:basics} we review $\gnd$ and its monomial characters;
in Section~\ref{s:main} we present our generating functions for these.

\section{$\gnd$ and its monomial characters}\label{s:basics}

The group $\gnd$ is generated by $n$ elements $s_1,\dotsc,s_{n-1}, t$
subject to the relations
\begin{equation}\label{eq:gnrelations}
  \begin{alignedat}2
    s_i^2 &= e &\quad &\text{for $i = 1, \dotsc, n-1$,}\\
    t^d &= e, &\quad &\\
    ts_1ts_1 &= s_1ts_1t, &\quad & \\
    s_is_j &= s_js_i &\quad &\text{for $|i-j| \geq 2$,}\\
    ts_j &= s_jt &\quad &\text{for $j \geq 2$,}\\
    s_is_js_i &= s_js_is_j &\quad &\text{for $|i-j| = 1$.}
  \end{alignedat}
\end{equation}
A one-line notation for elements of $\gnd$,
analogous to that for elements of $\sn$,
uses sequences of integer multiples of complex $d$th roots of unity.
Let $\zeta$ be a primitive $d$th root of unity, and
let $S$ be the set of sequences
\begin{equation}\label{eq:onelinenotation}
  \{ (\zeta^{\gamma_1} w_1, \dotsc, \zeta^{\gamma_n}w_n) \,|\,
  w_1 \cdots w_n \in \sn,
  (\gamma_1, \dotsc, \gamma_n) \in \mathbb Z/d\mathbb Z^n \}.
\end{equation}
We define an action of $\gnd$ on $S$ by letting the generators
act on a sequence $(a_1,\dotsc,a_n)$ as follows.
\begin{enumerate}
\item $s_i \circ ( a_1, \dotsc, a_n) =
  (a_1, \dotsc, a_{i-1}, a_{i+1}, a_i, a_{i+2}, \dotsc, a_n )$,
\item $t \circ (a_1, \dotsc, a_n) = (\zeta a_1, a_2, \dotsc, a_n )$.
\end{enumerate}
A bijection between $\gnd$ and $S$ is given by letting each element
$g \in \gnd$ act on the sequence $(1,\dotsc,n)$.
If $g \circ (1, \dotsc, n) = (\zeta^{\gamma_1}w_1, \dotsc, \zeta^{\gamma_n}w_n)$,
we define this second sequence to be the
{\em one-line notation} of $g$, and we write
$g = (\gamma,w)$, where
$\gamma = (\gamma_1, \dotsc, \gamma_n) \in \mathbb Z/d\mathbb Z^n$,
$w \in \sn$.
In particular, the identity element $e$ has one-line notation $1 \cdots n$.

Since $\gnd$ is a finite group, Brauer's Induced Character Theorem implies
that the set of monomial characters of $\gnd$ spans
{\em trace space} $\trsp(\gnd)$ of $\gnd$, the set of
all linear functionals
$\theta: \mathbb C[\gnd] \rightarrow \mathbb C$ satisfying
$\theta(gh) = \theta(hg)$
for all $g, h \in \gnd$.
(See, e.g., \cite{SnaithEBI}.)
This includes
all $\gnd$-characters.
$\trsp (\gnd)$ has dimension equal to the number of
conjugacy classes of $\gnd$,
equivalently, to the number of sequences
$\bos \lambda = (\lambda^{0},\dotsc,\lambda^{d-1})$
of $d$ (possibly empty)
integer partitions,
with
\begin{equation*}
  |\lambda^0| + \cdots + |\lambda^{d-1}| = n.
\end{equation*}
We call such a sequence a {\em $d$-partition} of $[n]$
and write $\bos \lambda \vdash n$.

In order to describe natural bases of $\trsp(\gnd)$,
we introduce certain
subgroups of $\gnd$ which are analogous to Young subgroups
of $\sn$.
%
Fix $d$-partition
$\bos \lambda = (\lambda^{0},\dotsc,\lambda^{d-1}) \vdash n$,
and define
$r_k = \ell(\lambda^k)$ for
$k = 0, \dotsc, d-1$.
We will say that an ordered set partition of $[n]$
of type
\begin{equation}\label{eq:osetcomp}
      (\lambda^0_1, \dotsc, \lambda^0_{r_0},
     \lambda^1_1, \dotsc, \lambda^1_{r_1},
     \dotsc,
     \lambda^{d-1}_1, \dotsc, \smash{\lambda^{d-1}_{r_{d-1}}}),
  \end{equation}
has {\em type $\bos\lambda$}.
In particular,
let $\mathbf K(\bos \lambda)
= (K_1^0,\dotsc, K_{r_0}^0,
  K_1^1,\dotsc,K_{r_1}^1,
  \dotsc, K_1^{d-1}, \dotsc,K_{r_{d-1}}^{d-1})$
be the ordered set partition of $[n]$ of type $\bos\lambda$
  whose blocks are the $r_0 + \cdots + r_{d-1}$ subintervals
  \begin{equation}\label{eq:Kdef}
    K^0_1 = [1,\lambda^0_1], \quad
    K^0_2 = [\lambda^0_1 +1, \lambda^0_1 + \lambda^0_2], \quad \dotsc, \quad
    K^{d-1}_{r_{d-1}} = [n-\lambda^{d-1}_{r_{d-1}}+1, n]
  \end{equation}
  of $[n]$.
  For $1 \leq i \leq j \leq n$, define the element
  $t_i = s_{i-1} \cdots s_1 t s_1 \cdots s_{i-1} \in \gnd$,
  and let $\mfg{[i,j]} \cong \mathbb Z/d\mathbb Z \wr \mfs{j-i+1}$
  be the
  subgroup of $\gnd$ generated by $\{ t_i, s_i, \dotsc, s_{j-1}\}$.
For $k = 0,\dotsc,d-1$, use (\ref{eq:Kdef}) to
define the subgroup
%
\begin{equation}\label{eq:cshiftyoung}
  \mfg{\bos \lambda, k} \defeq \mfg{K^k_1} \cdots \mfg{K^k_{r_k}}
  \cong \mfg{\lambda^k_1} \times \cdots \times \mfg{\lambda^k_{r_k}},
\end{equation}
of $\gnd$,
and finally
define the {\em Young subgroup}
\begin{equation}\label{eq:Youngdef}
  \mfg{\bos \lambda} \defeq
  \mfg{\bos \lambda, 0} \cdots \mfg{\bos \lambda, d-1} \cong
  \prod_{k=0}^{d-1}
  \big( \mfg{\lambda^k_1} \times \cdots \times \mfg{\lambda^k_{r_k}} \big)
\end{equation}
of $\gnd$.  Each element $y \in \gnd$ factors uniquely as $y_0 \cdots y_{d-1}$
with $y_k \in \mfg{\bos \lambda, k}$.

Several natural representations of $\gnd$ are defined by using
symmetric group representations and induction from $\mfg{\bos\lambda}$.
First, observe that
the subgroup of $\gnd$ generated by $s_1,\dotsc,s_{n-1}$ is isomorphic to $\sn$,
and that each $r$-dimensional $\sn$-representation $\rho$
can trivially be extended to a
$r$-dimensional $\gnd$-representation in at least $d$ ways:
by defining $\rho(t) = \zeta^k I$ for $k = 0,\dotsc,d-1$.
If the character of the $\sn$-representation is
$\chi$, call its extension $\delta_k \chi$.
Thus the two one-dimensional $\sn$-representations
\begin{equation*}
  \begin{alignedat}{2}
    1: s_i &\mapsto 1 &\qquad (w &\mapsto 1 \text{ for all $w \in \sn$}) ,\\
    \epsilon: s_i &\mapsto -1 &\qquad (w &\mapsto (-1)^{\inv(w)} \text{ for all $w \in \sn$})   
    \end{alignedat}
\end{equation*}
yield
$2d$ one-dimensional representations of $\gnd$:
\begin{equation}\label{eq:onediml}
  \begin{alignedat}{2}
    \delta_k: (s_i,t) &\mapsto (1, \zeta^k), &\qquad
    (g = (\gamma, w) &\mapsto (\gamma_1 \cdots \gamma_n)^k
    \text{ for all $g \in \gnd$}),\\
    \delta_k\epsilon: (s_i,t) &\mapsto (-1, \zeta^k), &\qquad
    (g = (\gamma, w) &\mapsto (-1)^{\inv(w)}(\gamma_1\cdots\gamma_n)^k
    \text{ for all $g \in \gnd$}),
  \end{alignedat}
\end{equation}
for $k = 0,\dotsc,d-1$.
Here, $\inv(w)$ denotes the Coxeter length of $w$.
(See, e.g., \cite[p.\,15]{BBCoxeter}.)
Next, observe that
for any $d$-tuple $(H_0, \dotsc, H_{d-1})$
of subgroups of a group $G$ which satisfy
\begin{equation}\label{eq:prodgroup}
  H \defeq H_0 \cdots H_{d-1} \cong
  H_0 \times \cdots \times H_{d-1},
  \end{equation}
and
characters $\theta_0, \dotsc, \theta_{d-1}$ of these,
we have that the function $\theta = \theta_0 \otimes \cdots \otimes \theta_{d-1}$
defined by $\theta(h_0 \cdots h_{d-1}) = \theta_0(h_0) \cdots \theta_{d-1}(h_{d-1})$
is a character of $H$, and ${\theta \upparrow_H^G}$ is a character of $G$.
In particular, 
the Young subgroup $\mfg{\bos\lambda}$ has the form (\ref{eq:prodgroup})
with $H_k = \mfg{\bos\lambda, k}$.
For every $d$-tuple
$\bos \beta = (\beta_0,\dotsc,\beta_{d-1}) \in \{1, \epsilon \}^d$
of one-dimensional symmetric group characters
we have
the one-dimensional $\mfg{\bos\lambda}$-character
\begin{equation}\label{eq:betadelta}
  \delta_0\beta_0 \otimes \cdots \otimes \delta_{d-1}\beta_{d-1},
\end{equation}
the corresponding monomial $\gnd$-character
\begin{equation}\label{eq:betadeltaind}
  \bos \beta^{\bos \lambda} \defeq
  (\delta_0\beta_0 \otimes
  \cdots \otimes \delta_{d-1}\beta_{d-1})
  \upparrow_{\mfg{\bos \lambda}}^{\gnd},
\end{equation}
and the basis $\{ \bos\beta^{\bos\lambda} \,|\, \bos\lambda \vdash n \}$
of $\trsp(\gnd)$.
The irreducible character basis
$\{ \chi^{\bos\lambda} \,|\, \bos\lambda \vdash n \}$
of $\trsp(\gnd)$ can be defined somewhat similarly.
Given $\bos\lambda = (\lambda^0, \dotsc, \lambda^{d-1}) \vdash n$,
define the $d$-partition
$\bos\lambda^{\ntnsp\bullet} = (|\lambda^0|, \dotsc, |\lambda^{d-1}|)$,
and
the $\mfg{\bos\lambda^{\ntnsp\bullet}}$-character
\begin{equation*}
  \delta_0\smash{\chi^{\lambda^0}} \ntnsp \otimes \cdots \otimes \delta_{d-1}\smash{\chi^{\lambda^{d-1}}},
  \end{equation*}
where $\smash{\chi^{\lambda^k}}$
is the irreducible
$\mfs{|\lambda^k|}$-character
indexed by the partition $\lambda^k$.
The corresponding induced
characters
\begin{equation}\label{eq:irrcharbasis}
  \chi^{\bos\lambda} =
  (\delta_0\smash{\chi^{\lambda^0}} \ntnsp \otimes \cdots \otimes \delta_{d-1}\smash{\chi^{\lambda^{d-1}}})
  \upparrow_{\mfg{\bos \lambda^{\bullet}}}^{\gnd}
\end{equation}
are the irreducible characters of $\gnd$.  (See, e.g., \cite[p.\,219]{AK}.)

For the purpose of creating generating functions for characters 
$\bos\beta^{\bos\lambda}$, it will be convenient to realize each as
the character of a submodule of $\mathbb C[\gnd]$, with $\gnd$ acting
by left multiplication.  To do this, we consider an arbitrary finite
group $G$, a subgroup $H$, an $H$-character $\theta$,
and the element
\begin{equation}
  T_H^\theta \defeq \sum_{h \in H} \theta(h^{-1})h \in \mathbb C[G].
\end{equation}
\begin{prop}\label{p:inductsum}
  Let $H$ be a subgroup of a finite group $G$ and let
  $\rho$ be a one-dimensional complex representation of $H$
  with character $\theta$ ($= \rho$).
  Let $U = (u_1,\dotsc,u_r)$
  be a transversal of representatives of cosets of $H$ in $G$.
  Let $G$ act by left multiplication on the submodule
  \begin{equation}\label{eq:vdefn}
  V := \spn_{\mathbb C} \{ u_i T_H^\theta \,|\, 1 \leq i \leq r \}
  \end{equation}
  of $\mathbb C[G]$.  Then $V$ is a $G$-module with 
  character $\theta \upparrow_H^G$.
\end{prop}
\begin{proof}
  To see that $V$ is a $G$-module,
  consider the action of $g \in G$ on the $j$th element of
  the defining basis of $V$.
  Let $u_iH$ be the unique coset satisfying
$gu_jH = u_iH$, i.e., $u_i^{-1}gu_j \in H$. Then we have
  \begin{equation}
  \begin{aligned}
  g u_j T_H^\theta &=
    g u_j \sum_{h\in H} \theta(h^{-1})h
  = u_i \sum_{h\in H} \theta(h^{-1}) u_i^{-1}gu_j h 
  = u_i \sum_{h' \in H} \theta( (h')^{-1}u_i^{-1}gu_j ) h' \\
  &= \theta(u_i^{-1}gu_j) u_i T^\theta_H,
  \end{aligned}
  \end{equation}
  since $\theta = \rho$
  is a homomorphism. 
It follows that in the $j$th column of the matrix representing $g$,
all components are $0$ except for the $i$th,
which is $\theta(u_i^{-1}gu_j)$.
But this is precisely the formula for entries of the matrix
$\rho \upparrow_H^G(g)$.  (See, e.g., \cite[Defn.\,1.12.2]{Sag}.)
\end{proof}

For $\chi = {\theta \upparrow_H^G}$, Proposition~\ref{p:inductsum}
allows us to express $T_G^\chi$ as a sum of conjugates of $T_H^\theta$.

\begin{lem}\label{l:indchargenfn}
  Let groups $G$, $H$, transversal $U = (u_1,\dotsc,u_r)$,
  $H$-character $\theta$, and $G$-module $V$ be as in 
  Proposition~\ref{p:inductsum},
  and let $A = (a_{i,j})$ 
  be the matrix of $g \in G$
  with respect to the defining basis (\ref{eq:vdefn}) of $V$.
  Then $a_{i,j}$ equals the coefficient of $g^{-1}$ in
    $u_j T_H^\theta u_i^{-1}$.
  In particular
  if $\chi$ is the character of $V$, then
  we have the identity
  \begin{equation}\label{eq:gpalggenfn}
    \sum_{i=1}^r u_i T_H^\theta u_i^{-1} = \sum_{g \in G} \chi(g) g^{-1}.
  \end{equation}
  in $\mathbb C[G]$.
\end{lem}
\begin{proof}
  By the proof of Proposition~\ref{p:inductsum},
  we have $a_{i,j} = \theta(u_i^{-1}gu_j)$
  if some $h \in H$ satisfies
  $g = u_ihu_j^{-1}$,
  and is $0$ otherwise.  On the other hand, we have
    \begin{equation}\label{eq:findg}
    u_j T_H^\theta u_i^{-1} = \sum_{h\in H} \theta(h^{-1}) u_jhu_i^{-1}.
  \end{equation}
    If there is no $h \in H$ satisfying $g^{-1} = u_jhu_i^{-1}$,
    then the coefficient of $g^{-1}$ in (\ref{eq:findg}) is $0$.
    Otherwise, the coefficient of $g^{-1}$ is
  \begin{equation*}
    \theta(h^{-1}) = \theta(u_i^{-1}gu_j).
  \end{equation*}
  It follows that $a_{i,j}$
  is equal to the coefficient of $g^{-1}$ in $u_jT_H^\theta u_i^{-1}$.
  Thus $\chi(g) = \sum_i a_{i,i}$
  is equal to the coefficient of $g^{-1}$ in $\sum_i u_i T_H^\theta u_i^{-1}$.
  \end{proof}

For
$G = \gnd$, $H = \mfg{\bos\lambda}$, and
$\theta$ as in (\ref{eq:betadelta}), 
the module
$V$
(\ref{eq:vdefn}) has a particularly nice form.
The element $T_H^\theta$ factors as
$T_{\mfg{\bos\lambda,0}}^{\delta_0\beta_0} \cdots T_{\mfg{\bos\lambda,d-1}}^{\delta_{d-1}\beta_{d-1}}$,
and each coset
$u\mfg{\bos\lambda}$
of $\mfg{\bos\lambda}$
has a unique representative
$g = (\gamma,w)$ satisfying $\gamma_1 = \cdots = \gamma_n = 0$ and
$w_i < w_{i+1}$ for $i, i+1$ belonging to the same block of
$\mathbf K(\bos \lambda)$,
i.e.,
\begin{equation}\label{eq:mincosetreps}
  w_1 < \cdots < w_{\lambda^0_1},
  \qquad
  w_{\lambda^0_1 + 1} < \cdots < w_{\lambda^0_1 + \lambda^0_2}, \dotsc,
  \qquad
  w_{n-\lambda^{d-1}_{r_{d-1}}+1} < \cdots < w_n.
  \end{equation}
Letting $\glambdamin$ be the set of such coset representatives,
we have
\begin{equation*}
  V = V(\bos\lambda,\bos\beta) =
  \spn_{\mathbb C} \{u T_{\mfg{\bos\lambda,0}}^{\delta_0\beta_0} \cdots
  T_{\mfg{\bos\lambda,d-1}}^{\delta_{d-1}\beta_{d-1}} \,|\, u \in \mfg{\bos\lambda}^- \},
\end{equation*}
and the following special case of Lemma~\ref{l:indchargenfn}.
\begin{cor}\label{c:indchargenfn}
  Fix a $d$-partition $\bos\lambda \vdash n$.
  For each one-dimensional $\mfg{\bos\lambda}$-character $\theta$
  of the form (\ref{eq:betadelta}),
  the monomial $\gnd$-character
  $\bos \beta^{\bos \lambda} = \theta \upparrow_{\,\mfg{\bos \lambda}}^{\,\gnd}$
  satisfies
  \begin{equation}\label{eq:indchargenfn}
    \sum_{u \in \mfg{\bos \lambda}^-} \nTksp u T_{\mfg{\bos \lambda}}^\theta u^{-1} =
    \sum_{g \in \gnd} \bos\beta^{\bos \lambda}(g^{-1}) g.
  \end{equation}
\end{cor}
For $d = 1,2$,
the group $\gnd$ (equal to the symmetric group or the hyperoctahedral group)
has real-valued irreducible characters.
Therefore each group element is conjugate to its inverse, and
the final sum of (\ref{eq:indchargenfn}) may be expressed as 
$\sum_{g \in \gnd} \bos\beta^{\bos\lambda}(g)g$.

\section{Main result}\label{s:main}

A generalization of the generating functions (\ref{eq:lmw1}) -- (\ref{eq:lmw2})
to monomial characters of $\gnd$ requires a polynomial ring and a
$|\gnd| = d^nn!$-dimensional subspace
analogous to the $n!$-dimensional span of the functions (\ref{eq:imm}).
Let $C_d = \{ \zeta^k \,|\, k \in \mathbb Z/d\mathbb Z \}$
be the subgroup of $\mathbb C$ consisting of $d$th roots of unity, and
for any subset $M \subseteq [n]$, let $C_dM$ be the complex numbers of the form
$\{ \zeta^km \,|\, k \in \mathbb Z/d\mathbb Z, m \in M \}$,
and define the set $x = \{ x_{i,j} \,|\, i \in [n], j \in C_d[n] \}$
of $dn^2$ variables.  One can think of $x$ as a collection of $d$
matrices of $n^2$ variables.
For example when $n = 2$ and $d = 3$, the variables are
\begin{equation}\label{eq:matrices}
  \begin{bmatrix} x_{1,1} & x_{1,2}\\ x_{2,1} & x_{2,2} \end{bmatrix},
  \qquad
  \begin{bmatrix} x_{1,\dt1} & x_{1,\dt2}\\ x_{2,\dt1} & x_{2,\dt2} \end{bmatrix},
  \qquad
  \begin{bmatrix} x_{1,\ddt1} & x_{1,\ddt2}\\ x_{2,\ddt1} & x_{2,\ddt2} \end{bmatrix},
\end{equation}
where we define
$\dt m \defeq \zeta m$, $\ddt m \defeq \zeta^2 m$
for variable subscripts $m = 1,2$, e.g.,
$x_{\smash{2,\ddt 1}} = x_{2, \zeta^2}$.

For $u \in \sn$, $g \in \gnd$,
write
\begin{equation*}
  x^{u,g} \defeq x_{u_1,g_1} \cdots x_{u_n,g_n},
\end{equation*}
and define the {\em $\gnd$-immanant subspace} of $\mathbb C[x]$ to be
\begin{equation*}
  \spn_{\mathbb C} \{ x^{e,g} = \permmon xg \,|\, g \in \gnd \}.
\end{equation*}
It is easy to see that these monomials satisfy
\begin{equation}\label{eq:ultor}
  x^{u,g} = x^{e,u^{-1}g}
\end{equation}
for all $u \in \sn$, $g \in \gnd$.
Thus for any fixed $u \in \sn$, the $\gnd$-immanant subspace
of $\mathbb C[x]$ may also be expressed as
$\spn_{\mathbb C} \{ x^{u,g} \,|\, g \in \gnd \}$.
The left- and right-regular representations of $\gnd$ define
left- and right-actions of $\gnd$ on the $\gnd$-immanant space,
\begin{equation}\label{eq:actonimm}
  h_1 \circ x^{e,g} \circ h_2 = x^{e,h_1gh_2},
\end{equation}
for $g, h_1, h_2 \in \gnd$.
For any function $\theta: \gnd \rightarrow \mathbb C$, define the
{\em type-$\gnd$ $\theta$-immanant} to be the generating function 
\begin{equation}\label{eq:gimmdef}
  \gimm{\theta}(x) = \ntksp \sum_{g \in \gnd} \ntksp \theta(g^{-1}) x^{e,g}
\end{equation}
for evaluations of $\theta$.
Our counterintuitive use of $g^{-1}$ in place of $g$ is necessitated by
Proposition~\ref{p:inductsum} -- Corollary~\ref{c:indchargenfn}.
(See also~\cite[Eq.\,(1)]{StemImm}.)
By the comment following Corollary~\ref{c:indchargenfn},
symmetric group
and hyperoctahedral group ($\bn \cong \mathbb Z/2\mathbb Z \wr \sn$)
immanants can be written
\begin{equation}\label{eq:hyperoctimm}
  \simm{\theta}(x) = \sum_{w \in \sn} \theta(w) x^{e,w},
  \qquad
  \bimm{\theta}(x) = \sum_{w \in \bn} \theta(w) x^{e,w}.
\end{equation}
For economy, we will generally supress $\sn$ from the notation of symmetric
group immanants.


Define the $d$ $n \times n$ matrices $Q_0(x), \dotsc, Q_{d-1}(x)$
by $Q_k(x) = (q_{i,j,k}(x))_{i,j \in [n]}$, where
\begin{equation}\label{eq:PandD2}
  q_{i,j,k}(x) = x_{i,j} + \zeta^{-k} x_{i,\zeta j} + \zeta^{-2k} x_{i,\zeta^2 j} +
  \cdots + \zeta^{-(d-1)k}x_{i,\zeta^{(d-1)}j}.
\end{equation}
The permanent and determinant of these matrices
are equal to $\gnd$-immanants for the one-dimensional characters
$\delta_0,\dotsc,\delta_{d-1}$, $\delta_0\epsilon, \dotsc, \delta_{d-1}\epsilon$
of $\gnd$.
Specifically, we have
\begin{equation}\label{eq:PandDimm2}
  \begin{gathered}
    \perm(Q_k(x)) =
    \nTksp \sum_{g = (\gamma, w) \in \gnd} \nTksp
    (\gamma_1 \cdots \gamma_n)^{-k} x^{e,g}
    = \gimm{\delta_k}(x)
    ,\\
    \det(Q_k(x)) =
    \nTksp \sum_{g = (\gamma, w) \in \gnd} \nTksp
    (-1)^{\inv(w)} (\gamma_1 \cdots \gamma_n)^{-k} x^{e,g}
    = \gimm{\delta_k\epsilon}(x).
  \end{gathered}
\end{equation}

More generally, we obtain
$\gnd$-analogs of the
Littlewood-Merris-Watkins
generating functions
(\ref{eq:lmw1}) -- (\ref{eq:lmw2})
by taking sums of products of immanants $\imm{\eta^\mu}^{\mfs m}$,
$\imm{\epsilon^\mu}^{\mfs m}$
of submatrices of $Q_0(x), \dotsc, Q_{d-1}(x)$,
where
$\eta^\mu = {1 \upparrow_{\mfs \mu}^{\mfs m}}$ and
$\epsilon^\mu = {\epsilon \upparrow_{\mfs \mu}^{\mfs m}}$ are
monomial
characters of $\mfs m$, for $m \leq n$.


\begin{thm}\label{t:main}
  Fix $d$-partition $\bos \lambda = (\lambda^0,\dotsc,\lambda^{d-1}) \vdash n$,
  and let
  $a_k = |\lambda^k|$,
  $r_k = \ell(\lambda^k)$.
  Fix character sequence
  $\bos \beta = (\beta_0,\dotsc,\beta_{d-1}) \in \{ 1, \epsilon \}^d$
  and define
  \begin{equation*}
    \beta_k^{\lambda^k} = \beta_k \upparrow_{\mfs{\lambda^k}}^{\mfs{a_k}}
    \in \{ \epsilon^{\lambda^k}, \eta^{\lambda^k} \}, \quad k = 0,\dotsc,d-1.
  \end{equation*}    
Then we have
\begin{equation}\label{eq:immprod3}
  \gimm{\bos \beta^{\bos \lambda}}(x) =
  \nTksp
  \sum_{(I_0,\dotsc,I_{d-1})}
  \nTksp
  \imm{\beta_0^{\lambda^0}}(Q_0(x)_{I_0,I_0})
    \cdots
  \imm{\beta_{d-1}^{\lambda^{d-1}}}(Q_{d-1}(x)_{I_{d-1},I_{d-1}}),
\end{equation}
%
where the sum is over all ordered set partitions of $[n]$
of type
$\bos \lambda^{\ntnsp\bullet} = (a_0,\dotsc,a_{d-1})$.

\end{thm}
\begin{proof}
  Define the
  $\mfg{\bos\lambda}$-character
  $\theta = \delta_0\beta_0 \otimes \cdots \otimes \delta_{d-1} \beta_{d-1}$
  and let
  $\bos\beta^{\bos\lambda} = \theta \upparrow_{\,\mfg{\bos \lambda}}^{\,\gnd}$.
  By Corollary~\ref{c:indchargenfn}, (\ref{eq:actonimm}), 
  and (\ref{eq:gimmdef}),
  we can express the left-hand side of (\ref{eq:immprod3}) as
    \begin{equation}\label{eq:lhsstimmconj}
      \sum_{g \in \gnd} \bos\beta^{\bos\lambda}(g^{-1}) \circ x^{e,g} =
    \sum_{g \in \gnd} \bos\beta^{\bos\lambda}(g^{-1}) g \circ x^{e,e} =
    \sum_{u \in \mfg{\bos \lambda}^-} u T_{\mfg{\bos \lambda}}^{\theta} u^{-1} \circ x^{e,e}.
    \end{equation}

    Now consider the right-hand side of (\ref{eq:immprod3}).
    By
    (\ref{eq:lmw1}) -- (\ref{eq:lmw2}),
    we may rewrite this as a sum of products of permanents and determinants,
  \begin{equation}\label{eq:lmwexpandQ}
    \sum_{\mathbf{J}}
    \bigg( \prod_{i=0}^{r_0} \imm{\beta_0}(Q_0(x)_{J^0_i,J^0_i}) \bigg)
    \cdots
    \bigg( \prod_{i=0}^{r_{d-1}} \imm{\beta_{d-1}}(Q_{d-1}(x)_{J^{d-1}_i,J^{d-1}_i}) \bigg),
  \end{equation}
  where the sum is over all ordered set partitions
    $\mathbf J =
    (J_1^0,\dotsc, J_{r_0}^0,
    \dotsc,
    J_1^{d-1}, \dotsc,J_{r_{d-1}}^{d-1})$
  of $[n]$ of type
  $\bos\lambda$, and where $\imm{\epsilon} = \det$, $\imm{1} = \perm$.
  For all $i,k$, the variables that appear in $Q_k(x)_{J^k_i,J^k_i}$
  are $x_{J^k_i,C_dJ^k_i}$.
  By (\ref{eq:PandDimm2}), we may again rewrite (\ref{eq:lmwexpandQ})
  as a sum 
  \begin{equation}\label{eq:cimmprodQ}
\sum_{\mathbf J}
    \bigg(
    \prod_{i=1}^{r_0}
    \grimm{\delta_0\beta_0}{\lambda^0_i}(x_{J^0_i,C_dJ^0_i})
    \bigg)
    \cdots
    \bigg(
    \prod_{i=1}^{r_{d-1}}
    \grimm{\delta_{d-1}\beta_{d-1}}{\lambda^{d-1}_i}(x_{J^{d-1}_i,C_dJ^{d-1}_i})
    \bigg)    
  \end{equation}
  in which each factor of each term has the form
  \begin{equation*}
    \grimm{\delta_k\beta_k}{\lambda^k_i}(x_{J^k_i,C_dJ^k_i}) =
    \begin{cases}
      {\displaystyle \sum_{g = (\gamma,w) \in \mfg{J^k_i}}} \nTksp
      (\gamma_1 \cdots \gamma_n)^{d-k} (x_{J^k_i,C_dJ^k_i})^{e,g}
      &\text{if $\beta_k = 1$},\\
      {\displaystyle \sum_{g = (\gamma,w) \in \mfg{J^k_i}}} \nTksp
      (\gamma_1 \cdots \gamma_n)^{d-k} (-1)^{\ell(w)}(x_{J^k_i,C_dJ^k_i})^{e,g}
      &\text{if $\beta_k = \epsilon$}.
      \end{cases}
  \end{equation*}      
  Define the set partition $\mathbf K = (K_1^0,\dotsc, K_{r_0}^0,
  \dotsc, K_1^{d-1}, \dotsc,K_{r_{d-1}}^{d-1})$ of type $\bos\lambda$
  as in (\ref{eq:Kdef}),
  and for each ordered set partition $\mathbf J$
  of type $\bos\lambda$
  define $u = u(\mathbf J) \in \mfg{\bos \lambda}^-$ to be the element whose
  one-line notation has the $\lambda^k_i$
  consecutive letters $K^k_i$
  in positions $J^k_i$,
  for $k = 0, \dotsc, d-1$ and $i = 1,\dotsc,r_k$.
  In particular,
  $u^{-1}$ is the element in $\sn \subset \gnd$
  whose one-line notation contains the increasing rearrangement of $J^k_i$
  in the consecutive positions $K^k_i$
  for $k = 0, \dotsc, d-1$ and $i = 1,\dotsc,r_k$.
  By (\ref{eq:mincosetreps}),
  the map $\mathbf J \mapsto u(\mathbf J)$
  defines a bijective correspondence between ordered set partitions
  of type $\bos \lambda$ and $\mfg{\bos \lambda}^-$.
  Thus in the expansion of the product (\ref{eq:cimmprodQ}),
  the monomials which appear are precisely the set
  $\{ x^{u^{-1}\ntnsp,yu^{-1}} \,|\, y \in \mfg{\bos \lambda} \}$.
  Factoring $y = y_0 \cdots y_{d-1}$ with
  $y_k \in \mfg{\bos \lambda, k}$,
  we may express the coefficient of each such monomial as
  \begin{equation}\label{eq:thetaeval}
\delta_0\beta_0(y_0^{-1}) \cdots \delta_{d-1}\beta_{d-1}(y_{d-1}^{-1}) =
    \theta(y^{-1}).
    \end{equation}
  Using these facts and (\ref{eq:ultor}), (\ref{eq:actonimm}),
we may rewrite (\ref{eq:lmwexpandQ}) as
\begin{equation*}
  \sum_{u \in \mfg{\bos \lambda}^-} \sum_{y \in \mfg{\bos \lambda}} \ntksp
  \theta(y^{-1}) x^{u^{-1},yu^{-1}}
  = \sum_{u \in \mfg{\bos \lambda}^-} \sum_{y \in \mfg{\bos \lambda}} \ntksp
  \theta(y^{-1}) uyu^{-1} \circ x^{e,e}
  = \sum_{u \in \mfg{\bos \lambda}^-} \ntksp uT_{\mfg{\bos \lambda}}^\theta 
  u^{-1} \circ x^{e,e}
  \end{equation*}
to see that it is equal to (\ref{eq:lhsstimmconj}).
\end{proof}

We illustrate with an example.
Consider the group $\mfg{6,3} = \mathbb Z/3\mathbb Z \wr \mfs 6$.
It trace space $\trsp(\mfg{6,3})$ has dimension equal to the
number of $3$-partitions of $6$, and its immanant space
\begin{equation*}
\spn_{\mathbb C} \{ x_{1,g_1} \cdots x_{6,g_6} \,|\, (g_1,\dotsc,g_6) \in \mfg{6,3} \}
\end{equation*}
requires the $6^2 \cdot 3 = 108$ variables
$\{ x_{i,\zeta^k j} \,|\, i,j = 1,\dotsc,6; k = 0,\dotsc,2 \}$,
where $\zeta = e^{2\pi i/3}$.
To economize notation, we define
$\dt m \defeq \zeta m$, $\ddt m \defeq \zeta^2 m$,
as in (\ref{eq:matrices}).
The $2^3 = 8$ monomial character bases
correspond to the triples 
of one-dimensional symmetric group characters
$(1,1,1), (1,1,\epsilon), (1, \epsilon, 1), \dotsc,
(\epsilon,\epsilon,\epsilon)$,
so that the basis corresponding to $(\epsilon,\epsilon,1)$ is
\begin{equation}\label{eq:basisex}
  \big \{ (\epsilon,\epsilon,1)^{\bos\lambda} =
  (\epsilon \otimes \delta_1\epsilon \otimes \delta_2)
  \upparrow_{\,\mfg{\bos\lambda}}^{\,\mfg{6,3}} \, \big | \, \bos\lambda \vdash 6 \big \}.
\end{equation}
Consider
the basis element $(\epsilon,\epsilon,1)^{(21, 1, 2)}$.
To evaluate $(\epsilon,\epsilon,1)^{(21, 1, 2)}(g)$ for all $g \in \gnd$,
we write its immanant
$\imm{(\epsilon,\epsilon,1)^{(21,1,2)}}^{\mfg{6,3}}(x)$
as a sum of $60$ terms
\begin{equation}\label{eq:21,1,2}
  \begin{gathered}
    \imm{\epsilon^{21}}(Q_0(x)_{123,123})\imm{\epsilon^1}(Q_1(x)_{4,4})
    \imm{\eta^2}(Q_2(x)_{56,56}) \\
    + \imm{\epsilon^{21}}(Q_0(x)_{123,123})\imm{\epsilon^1}(Q_1(x)_{5,5})
    \imm{\eta^2}(Q_2(x)_{46,46}) \\
    + \imm{\epsilon^{21}}(Q_0(x)_{123,123})\imm{\epsilon^1}(Q_1(x)_{6,6})
    \imm{\eta^2}(Q_2(x)_{45,45}) \\
    + \imm{\epsilon^{21}}(Q_0(x)_{124,124})\imm{\epsilon^1}(Q_1(x)_{3,3})
    \imm{\eta^2}(Q_2(x)_{56,56}) \\
    \vdots \\
    + \imm{\epsilon^{21}}(Q_0(x)_{456,456})\imm{\epsilon^1}(Q_1(x)_{3,3})
  \imm{\eta^2}(Q_2(x)_{12,12}),
  \end{gathered}
\end{equation}
each corresponding to an ordered set partition of $[6]$
of type $(3,1,2)$.
Consider the term corresponding to the ordered set partition $(136, 4, 25)$.
It is a product of the three factors

\begin{equation}\label{eq:136|4|25}
  \begin{aligned}
    \imm{\epsilon^{21}}(Q_0(x)_{136,136}) &=
  \det \ntnsp \begin{bmatrix}
    x_{1,1} + x_{1, \dt 1} + x_{1, \ddt 1} & x_{1,3} + x_{1, \dt 3} + x_{1, \ddt 3} \\
    x_{3,1} + x_{3, \dt 1} + x_{3, \ddt 1} & x_{3,3} + x_{3, \dt 3} + x_{3, \ddt 3}
  \end{bmatrix}
  \ntnsp (x_{6,6} + x_{6, \dt 6} + x_{6, \ddt 6})\\
  &~+
  \det \ntnsp \begin{bmatrix}
    x_{1,1} + x_{1, \dt 1} + x_{1, \ddt 1} & x_{1,6} + x_{1, \dt 6} + x_{1, \ddt 6} \\
    x_{6,1} + x_{6, \dt 1} + x_{6, \ddt 1} & x_{6,6} + x_{6, \dt 6} + x_{6, \ddt 6}
  \end{bmatrix}
  \ntnsp (x_{3,3} + x_{3, \dt 3} + x_{3, \ddt 3})\\
  &~+
  \det \ntnsp \begin{bmatrix}
    x_{3,3} + x_{3, \dt 3} + x_{3, \ddt 3} & x_{3,6} + x_{3, \dt 6} + x_{3, \ddt 6} \\
    x_{6,3} + x_{6, \dt 3} + x_{6, \ddt 3} & x_{6,6} + x_{6, \dt 6} + x_{6, \ddt 6}
  \end{bmatrix}
  \ntnsp (x_{1,1} + x_{1, \dt 1} + x_{1, \ddt 1}),\\
  \imm{\epsilon^{1}}(Q_1(x)_{4,4}) &=
  x_{4,4} + \zeta^2 x_{4, \dt 4} + \zeta x_{4, \ddt 4},\\
  \imm{\eta^{2}}(Q_2(x)_{25,25}) &=
  \perm \ntnsp \begin{bmatrix}
    x_{2,2} + \zeta x_{2,\dt 2} + \zeta^2 x_{2,\ddt 2} &
    x_{2,5} + \zeta x_{2,\dt 5} + \zeta^2 x_{2,\ddt 5} \\
    x_{5,2} + \zeta x_{5,\dt 2} + \zeta^2 x_{5,\ddt 2} &
    x_{5,5} + \zeta x_{5,\dt 5} + \zeta^2 x_{5,\ddt 5}
  \end{bmatrix}\ntnsp.
  \end{aligned}
\end{equation}
It is easy to see that this term, like all others in (\ref{eq:21,1,2}),
contributes
$3$ to the coefficient of $x_{1,1}x_{2,2}x_{3,3}x_{4,4}x_{5,5}x_{6,6}$.
Thus we have
$(\epsilon,\epsilon,1)^{(21,1,2)}(123456) = 180$.  Now consider the computation of
$(\epsilon,\epsilon,1)^{(21,1,2)}(623\dt45\dt1)$.
The term (\ref{eq:136|4|25}) contributes $-\zeta^2$ to the
coefficient of $x_{1,6}x_{2,2}x_{3,3}x_{\smash{4,\dt4}}x_{5,5}x_{\smash{6,\dt1}}$,
as do the terms
in (\ref{eq:21,1,2}) corresponding to the other
two ordered set partitions $(1a6,4,bc)$. 
The term corresponding to
the ordered set partition
$(235,4,16)$
contributes $3 \zeta^2\zeta = 3$, and the three terms corresponding to
the ordered set partitions
$(ab4,c,16)$ contribute $3 \zeta$.
Terms corresponding to all other ordered set partitions contribute $0$.
Thus we have
\begin{equation*}
  (\epsilon,\epsilon,1)^{(21,1,2)}(623\dt45\dt1) = 3(1 + \zeta - \zeta^2).
\end{equation*}

It would be interesting to extend Theorem~\ref{t:main} to
obtain a generating function for the monomial characters of Hecke algebras
of wreath products~\cite{AK}, as was done for monomial characters
of the Hecke algebra of $\sn$ in \cite[Thm.\,2.1]{KSkanQGJ}.

\end{document}